\documentclass[12pt,twoside]{article}
\usepackage{amsmath,amsthm,amssymb,amscd,ascmac, amsfonts}
\usepackage{mathrsfs}
\usepackage{stmaryrd}
\usepackage{braket}
\usepackage{accents}
\usepackage{url}

\numberwithin{equation}{section}
\usepackage[dvips]{graphicx,color,psfrag}

\makeatletter

\renewcommand{\Re}{\mathrm{Re}\,}
\renewcommand{\Im}{\mathrm{Im}\,}

\newcommand{\sn}{\mathrm{sn}\,}
\newcommand{\cn}{\mathrm{cn}\,}
\newcommand{\dn}{\mathrm{dn}\,}
\newcommand{\cs}{\mathrm{cs}\,}
\newcommand{\ds}{\mathrm{ds}\,}
\newcommand{\ns}{\mathrm{ns}\,}

\renewcommand{\mod}{\,\mathrm{mod}\,}

\newcommand{\SL}{\mathrm{SL}}

\newtheorem*{multitheorem}{\variable@name}

\theoremstyle{definition}
\newcommand{\variable@name}{Theorem}
\newtheorem*{multiproclaim}{\variable@name}

\theoremstyle{plain}
\newtheorem{thm}{Theorem}

\newtheorem{lem}[thm]{Lemma}
\newtheorem{cor}[thm]{Corollary}

\theoremstyle{definition}

\newtheorem{rmk}[thm]{Remark}

\begin{document}
\title{An elliptic analogue of Fukuhara's trigonometeric identities}
\author{Genki Shibukawa}
\date{
\small MSC classes\,:\,11F11, 11F20, 11M36, 33E05}
\pagestyle{plain}

\maketitle


\begin{abstract}
We obtain new elliptic function identities, which are an elliptic analogue of Fukuhara's trigonometric identities. 
We show that the coefficients of Laurent expansions at $z=0$ of our elliptic identities give rise to some reciprocity laws for elliptic Dedekind sums. 
\end{abstract}

\section{Introduction}
Our starting point is the following identities. 
Let $a$ and $b$ be relatively prime positive integers and $z \in \mathbb{C}$. \\
\noindent
{\rm{(0)}}\,{\rm{((1.1) in \cite{F})}}
For any complex number $z$, 
\begin{align}
& \frac{1}{a}\sum_{\nu =1}^{a-1}\cot{\left(\frac{\pi b \nu }{a}\right)}\cot{\left(\pi \left(z+\frac{\nu}{a}\right)\right)} 
   +\frac{1}{b}\sum_{\nu =1}^{b-1}\cot{\left(\frac{\pi a \nu}{b}\right)}\cot{\pi \left(\left(z+\frac{\nu}{b}\right)\right)} \nonumber \\
   \label{eq:prot formula}
   & \quad =
   -\cot{(\pi az)}\cot{(\pi bz)}+\frac{1}{ab}\csc(\pi z)^{2}-1.
\end{align}

\noindent
{\rm{(1)}}\,{\rm{((1.2) in \cite{F})}} 
If $a$ is even, then
\begin{align}
& \frac{1}{a}\sum_{\nu =1}^{a-1}(-1)^{\nu}\cot{\left(\frac{\pi b \nu }{a}\right)}\cot{\left(\pi \left(z+\frac{\nu}{a}\right)\right)} 
   +\frac{1}{b}\sum_{\nu =1}^{b-1}\csc{\left(\frac{\pi a \nu}{b}\right)}\cot{\left(\pi \left(z+\frac{\nu}{b}\right)\right)} \nonumber \\
\label{eq:another prot thm1.3}
   & \quad =
   -\csc{(\pi az)}\cot{(\pi bz)}+\frac{1}{ab}\csc(\pi z)^{2}. 
\end{align}

\noindent
{\rm{(2)}}\,{\rm{((1.4) in \cite{F})}} 
If $a$ is odd, then
\begin{align}
& \frac{1}{a}\sum_{\nu =1}^{a-1}(-1)^{\nu}\cot{\left(\frac{\pi b \nu }{a}\right)}\csc{\left(\pi \left(z+\frac{\nu}{a}\right)\right)}
   +\frac{1}{b}\sum_{\nu =1}^{b-1}\csc{\left(\frac{\pi a \nu}{b}\right)}\csc{\left(\pi \left(z+\frac{\nu}{b}\right)\right)} \nonumber \\
\label{eq:another prot thm1.4}
   & \quad =
   -\csc{(\pi az)}\cot{(\pi bz)}+\frac{1}{ab}\csc(\pi z)\cot(\pi z). 
\end{align}

\noindent
{\rm{(3)}}\,{\rm{((1.3) in \cite{F})}} 
If $a+b$ is even, then 
\begin{align}
& \frac{1}{a}\sum_{\nu =1}^{a-1}(-1)^{\nu}\csc{\left(\frac{\pi b \nu }{a}\right)}\cot{\left(\pi \left(z+\frac{\nu}{a}\right)\right)}
   +\frac{1}{b}\sum_{\nu =1}^{b-1}(-1)^{\nu}\csc{\left(\frac{\pi a \nu}{b}\right)}\cot{\left(\pi \left(z+\frac{\nu}{b}\right)\right)} \nonumber \\
\label{eq:another prot thm1.1}
& \quad =
   -\csc{(\pi az)}\csc{(\pi bz)}+\frac{1}{ab}\csc(\pi z)^{2}. 
\end{align}

\noindent
{\rm{(4)}}\,{\rm{((1.5) in \cite{F})}} 
If $a+b$ is odd, then
\begin{align}
& \frac{1}{a}\sum_{\nu =1}^{a-1}(-1)^{\nu}\csc{\left(\frac{\pi b \nu }{a}\right)}\csc{\left(\pi \left(z+\frac{\nu}{a}\right)\right)}
   +\frac{1}{b}\sum_{\nu =1}^{b-1}(-1)^{\nu}\csc{\left(\frac{\pi a \nu}{b}\right)}\csc{\left(\pi \left(z+\frac{ \nu}{b}\right)\right)} \nonumber \\
\label{eq:another prot thm1.2}
& \quad =
   -\csc{(\pi az)}\csc{(\pi bz)}+\frac{1}{ab}\csc(\pi z)\cot(\pi z).
\end{align}

The formula (\ref{eq:prot formula}) was given by Fukuhara \cite{F}. 
Precisely, in \cite{F}, Fukuhara pointed out that (\ref{eq:prot formula}) is derived from specialization of Dieter's formula (Theorem\,2.4 of \cite{D}) and proved (\ref{eq:another prot thm1.3}) - (\ref{eq:another prot thm1.2}). 
Further, he compared the coefficients of Laurent expansions at $z=0$ of identities (\ref{eq:prot formula}) - (\ref{eq:another prot thm1.2}) and obtained reciprocity laws for Dedekind-Apostol sums. 
The most simplest case of his reciprocity laws are the following. 

\noindent
{\rm{(0)}}\,{\rm{((1.11) in \cite{F})}}
If $a$ and $b$ are relatively prime positive integers, then 
\begin{align}
& \frac{1}{a}\sum_{\nu =1}^{a-1}\cot{\left(\frac{\pi b \nu }{a}\right)}\cot{\left(\frac{\pi \nu}{a}\right)}
+\frac{1}{b}\sum_{\nu =1}^{b-1}\cot{\left(\frac{\pi a \nu}{b}\right)}\cot{\left(\frac{\pi \nu}{b}\right)} \nonumber \\
\label{eq:2prot formula} 
& \quad =
   \frac{a^{2}+b^{2}+1-3ab}{3ab}.
\end{align}

\noindent
{\rm{(1)}}\,{\rm{((1.12) in \cite{F})}} 
If $a$ is even, then
\begin{align}
& \frac{1}{a}\sum_{\nu =1}^{a-1}(-1)^{\nu}\cot{\left(\frac{\pi b \nu }{a}\right)}\cot{\left(\frac{\pi \nu}{a}\right)}
+\frac{1}{b}\sum_{\nu =1}^{b-1}\csc{\left(\frac{\pi a \nu}{b}\right)}\cot{\left(\frac{\pi \nu}{b}\right)} \nonumber \\
\label{eq:2another prot thm1.3}
& \quad =
   \frac{-a^{2}+2b^{2}+2}{6ab}.
\end{align}

\noindent
{\rm{(2)}}\,{\rm{((1.14) in \cite{F})}} 
If $a$ is odd, then
\begin{align}
& \frac{1}{a}\sum_{\nu =1}^{a-1}(-1)^{\nu}\cot{\left(\frac{\pi b \nu }{a}\right)}\csc{\left(\frac{\pi \nu}{a}\right)}
   +\frac{1}{b}\sum_{\nu =1}^{b-1}\csc{\left(\frac{\pi a \nu}{b}\right)}\csc{\left(\frac{\pi \nu}{b}\right)} \nonumber \\
\label{eq:2another prot thm1.4}
& \quad =
   \frac{-a^{2}+2b^{2}-1}{6ab}.
\end{align}

\noindent
{\rm{(3)}}\,{\rm{((1.13) in \cite{F})}} 
If $a+b$ is even, then 
\begin{align}
& \frac{1}{a}\sum_{\nu =1}^{a-1}(-1)^{\nu}\csc{\left(\frac{\pi b \nu }{a}\right)}\cot{\left(\frac{\pi \nu}{a}\right)}
+\frac{1}{b}\sum_{\nu =1}^{b-1}(-1)^{\nu}\csc{\left(\frac{\pi a \nu}{b}\right)}\cot{\left(\frac{\pi \nu}{b}\right)} \nonumber \\
\label{eq:2another prot thm1.1}
& \quad =
   \frac{-a^{2}-b^{2}+2}{6ab}.
\end{align}

\noindent
{\rm{(4)}}\,{\rm{((1.15) in \cite{F})}} 
If $a+b$ is odd, then
\begin{align}
& \frac{1}{a}\sum_{\nu =1}^{a-1}(-1)^{\nu}\csc{\left(\frac{\pi b \nu }{a}\right)}\csc{\left(\frac{\pi \nu}{a}\right)}
+\frac{1}{b}\sum_{\nu =1}^{b-1}(-1)^{\nu}\csc{\left(\frac{\pi a \nu}{b}\right)}\csc{\left(\frac{\pi \nu}{b}\right)} \nonumber \\
\label{eq:2another prot thm1.2}
& \quad =
   \frac{-a^{2}-b^{2}-1}{6ab}.
\end{align}

On the other hand, Fukuhara and Yui obtained the following elliptic function identity (Theorem\,2.1 in \cite{FY}) which is regarded as an elliptic analogue of the trigonometric identity (\ref{eq:prot formula}). 
If $a+b$ is odd, then 
\begin{align}
& \frac{1}{a}
   \!\!\!\!\!\sum_{\substack{\mu,\nu =0 \\ (\mu,\nu) \not=(0,0)}}^{a-1}\!\!\!\!\!
   (-1)^{\mu }\cs\left(2Kb\frac{\mu \tau +\nu}{a},k \right)\cs\left(2K\left(z+\frac{\mu \tau +\nu}{a}\right),k\right) \nonumber \\
& \quad +\frac{1}{b}
   \!\!\!\!\!\sum_{\substack{\mu,\nu =0 \\ (\mu,\nu ) \not=(0,0)}}^{b-1}\!\!\!\!\!
   (-1)^{\mu }\cs\left(2Ka\frac{\mu \tau +\nu}{b},k \right)\cs\left(2K\left(z+\frac{\mu \tau +\nu}{b}\right),k\right) \nonumber \\
\label{eq:elliptic prot formula}
   & \quad =
   -\cs(2Ka z,k)\cs(2Kb z,k)+\frac{1}{ab}\ds(2Kz,k)\ns(2Kz,k),
\end{align}
where $\cs(z,k)$ is the Jacobi elliptic function (see Section\,2). 
From this elliptic function identity, Fukuhara and Yui also gave reciprocity laws (Theorem\,2.2 in \cite{FY}) for the elliptic Dedekind-Apostol sums
$$
\frac{1}{b}\sum_{\substack{\mu,\nu =0 \\ (\mu,\nu ) \not=(0,0)}}^{b-1}
   (-1)^{\mu }
   \cs\left(2Ka\frac{\mu \tau +\nu }{b},k\right)
   \cs^{(N)}\left(2K\frac{\mu \tau +\nu }{b},k\right),
$$
where $\cs^{(N)}(z)$ is the $N$-th derivative of the $\cs(z)$. 
The simplest case of Fukuhara-Yui's reciprocity is the following (Lemma\,3.1 in \cite{FY})
\begin{align}
& \frac{1}{a}
   \!\!\!\!\!\sum_{\substack{\mu,\nu =0 \\ (\mu,\nu) \not=(0,0)}}^{a-1}\!\!\!\!\!
   (-1)^{\mu }\cs\left(2Kb\frac{\mu \tau +\nu}{a},k \right)\cs\left(2K\left(\frac{\mu \tau +\nu}{a}\right),k\right) \nonumber \\
& \quad +\frac{1}{b}
   \!\!\!\!\!\sum_{\substack{\mu,\nu =0 \\ (\mu,\nu ) \not=(0,0)}}^{b-1}\!\!\!\!\!
   (-1)^{\mu }\cs\left(2Ka\frac{\mu \tau +\nu}{b},k \right)\cs\left(2K\left(\frac{\mu \tau +\nu}{b}\right),k\right) \nonumber \\
\label{eq:elliptic prot formula2}
   & \quad =
   \frac{a^{2}+b^{2}+1}{6ab}\left(2-\lambda (\tau )\right),
\end{align}
which is an elliptic analogue of the reciprocity law (\ref{eq:2prot formula}) for the classical Dedekind sum
$$
s(a;b)
   :=
   \frac{1}{4b}\sum_{\nu =1}^{b-1}\cot{\left(\frac{\pi a \nu}{b}\right)}\cot{\left(\frac{\pi \nu}{b}\right)}. 
$$

In this article, we give an elliptic analogue of (\ref{eq:another prot thm1.3}) - (\ref{eq:another prot thm1.2}) and (\ref{eq:2another prot thm1.3}) - (\ref{eq:2another prot thm1.2}). 
The content of this paper is as follows. 
In Section\,2, we introduce the elliptic functions $\cs(z,k)$, $\ds(z,k)$, $\ns(z,k)$, and list their fundamental properties. 
Section\,3 is the main part of this article and we prove our main results (Theorem\,\ref{thm:main theorem1}, Theorem\,\ref{thm:main theorem2} and Corollary\,\ref{thm:main theorem3}). 
In Section\,4, we give all the examples of Theorem\,\ref{thm:main theorem1} and Corollary\,\ref{thm:main theorem3}. 

\section{Preliminaries}
Throughout the paper, we denote the ring of rational integers by $\mathbb{Z}$, 
the field of real numbers by $\mathbb{R}$, the field of complex numbers by $\mathbb{C}$ and the upper half plane $\mathfrak{H}:=\{z \in \mathbb{C}\mid \Im{z}>0\}$. 
For $\tau \in \mathfrak{H}$, we put   
$$
e(x):=e^{2\pi \sqrt{-1}x}, \quad q:=e(\tau).
$$
First, we recall the Jacobi theta functions 
\begin{align}
\theta_{1}(z,\tau)
   :=&
   2\sum_{n=0}^{\infty}(-1)^{n}q^{\frac{1}{2}(n+\frac{1}{2})^{2}}\sin{((2n+1)\pi z)} \nonumber \\
   =&
   2q^{\frac{1}{8}}\sin{\pi z}\prod_{n=1}^{\infty}(1-q^{n})(1-q^{n}e(z))(1-q^{n}e(-z)), \nonumber \\
\theta_{2}(z,\tau)
   :=&
   2\sum_{n=0}^{\infty}q^{\frac{1}{2}(n+\frac{1}{2})^{2}}\cos{((2n+1)\pi z)} \nonumber \\
   =&
   2q^{\frac{1}{8}}\cos{\pi z}\prod_{n=1}^{\infty}(1-q^{n})(1+q^{n}e(z))(1+q^{n}e(-z)), \nonumber \\
\theta_{3}(z,\tau)
   :=&
   1+2\sum_{n=1}^{\infty}q^{\frac{1}{2}n^{2}}\cos{(2n\pi z)} \nonumber \\
   =&
   \prod_{n=1}^{\infty}(1-q^{n})(1+q^{n-\frac{1}{2}}e(z))(1+q^{n-\frac{1}{2}}e(-z)), \nonumber \\
\theta_{4}(z,\tau)
   :=&
   1+2\sum_{n=1}^{\infty}(-1)^{n}q^{\frac{1}{2}n^{2}}\cos{(2n\pi z)} \nonumber \\
   =&
   \prod_{n=1}^{\infty}(1-q^{n})(1-q^{n-\frac{1}{2}}e(z))(1-q^{n-\frac{1}{2}}e(-z)). \nonumber
\end{align}
Further we put 
$$
k=k(\tau):=\frac{\theta_{2}(0,\tau)^{2}}{\theta_{3}(0,\tau)^{2}},\quad 
\lambda =\lambda (\tau):=k(\tau)^{2},\quad 
K=K(\tau):=\frac{\pi}{2}\theta_{3}(0,\tau)^{2}
$$
and introduce the Jacobi elliptic functions
\begin{align}
\sn(2Kz,k)
   :=&
   \frac{\theta_{3}(0,\tau)}{\theta_{2}(0,\tau)}\frac{\theta_{1}(z,\tau)}{\theta_{4}(z,\tau)}, \nonumber \\
\cn(2Kz,k)
   :=&
   \frac{\theta_{4}(0,\tau)}{\theta_{2}(0,\tau)}\frac{\theta_{2}(z,\tau)}{\theta_{4}(z,\tau)}, \nonumber \\
\dn(2Kz,k)
   :=&
   \frac{\theta_{4}(0,\tau)}{\theta_{3}(0,\tau)}\frac{\theta_{3}(z,\tau)}{\theta_{4}(z,\tau)}. \nonumber 
\end{align}
As is well known, the Jacobi elliptic functions $\sn(2Kz,k)$, $\cn(2Kz,k)$ and $ \dn(2Kz,k)$ only depend on $\lambda (\tau)=k(\tau)^{2}$ (elliptic lambda function) that is a modular function of the modular subgroup
$$
\Gamma (2)
   :=
   \left\{
   \begin{pmatrix}
   a & b \\
   c & d
   \end{pmatrix} \in \SL_{2}(\mathbb{Z}) \,\Bigg\vert \,a\equiv d\equiv 1 \,(\mod 2),\,\,b\equiv c\equiv 0 \,(\mod 2)\right\}.
$$
Therefore under the following we restrict $\tau$ to the fundamental domain of $\Gamma (2)$ 
$$
\Gamma (2)\setminus \mathfrak{H}
   \simeq \left\{\tau \in \mathfrak{H} \,\Bigg\vert \, |\Re{\tau}|\leq 1, \, \left|\tau\pm \frac{1}{2}\right|\geq \frac{1}{2}\right\}.
$$


The elliptic functions $\cs(2Kz,k)$, $\ds(2Kz,k)$ and $\ns(2Kz,k)$ are defined by
\begin{align}
\cs(2Kz,k)&:=\frac{\cn(2Kz,k)}{\sn(2Kz,k)}, \nonumber \\
\ds(2Kz,k)&:=\frac{\dn(2Kz,k)}{\sn(2Kz,k)}, \nonumber \\
\ns(2Kz,k)&:=\frac{1}{\sn(2Kz,k)}. \nonumber 
\end{align}
The elliptic function $\cs(2Kz,k)$ is regarded as an elliptic analogue of $\cot{(\pi z)}=\frac{\cos{(\pi z)}}{\sin{(\pi z)}}$. 
Similarly, $\ds(2Kz,k)$ and $\ns(2Kz,k)$ are regarded as an elliptic analogue of $\csc{(\pi z)}=\frac{1}{\sin{(\pi z)}}$. 
According to the wolfram functions site \cite{Wo} and \cite{W}, we list fundamental properties of $\cs(2Kz,k)$, $\ds(2Kz,k)$ and $\ns(2Kz,k)$. 
\begin{lem}
{\rm{(1)}}\;(parity) 
\begin{align}
\label{eq:parity of cs}
\cs(-2Kz,k)&=-\cs(2Kz,k), \\
\label{eq:parity of ds}
\ds(-2Kz,k)&=-\ds(2Kz,k), \\
\label{eq:parity of ns}
\ns(-2Kz,k)&=-\ns(2Kz,k).
\end{align}
\url{http://functions.wolfram.com/EllipticFunctions/JacobiCS/04/02/01/}\\
\url{http://functions.wolfram.com/EllipticFunctions/JacobiDS/04/02/01/}\\
\url{http://functions.wolfram.com/EllipticFunctions/JacobiNS/04/02/01/}\\
\noindent
{\rm{(2)}}\;(periodicity) For any $\mu,\nu \in \mathbb{Z}$, 
\begin{align}
\label{eq:periodicity of cs}
\cs(2K(z+\mu \tau+\nu ),k)&=(-1)^{\mu }\cs(2Kz,k), \\
\label{eq:periodicity of ds}
\ds(2K(z+\mu \tau+\nu ),k)&=(-1)^{\mu +\nu }\ds(2Kz,k), \\
\label{eq:periodicity of ns}
\ns(2K(z+\mu \tau+\nu ),k)&=(-1)^{\nu }\ns(2Kz,k).
\end{align}
\url{http://functions.wolfram.com/EllipticFunctions/JacobiCS/04/02/03/}\\
\url{http://functions.wolfram.com/EllipticFunctions/JacobiDS/04/02/03/}\\
\url{http://functions.wolfram.com/EllipticFunctions/JacobiNS/04/02/03/}\\
\noindent
{\rm{(3)}}\;(Laurent expansions at $z=0$)  
\begin{align}
\label{eq:Laurent expansion of cs}
2K\cs(2Kz,k)&=\frac{1}{z}+\left(-\frac{1}{3}+\frac{1}{6}\lambda \right)(2K)^{2}z+\left(-\frac{1}{45}+\frac{1}{45}\lambda +\frac{7}{360}\lambda ^{2}\right)(2K)^{4}z^{3}+\cdots, \\
\label{eq:Laurent expansion of ds}
2K\ds(2Kz,k)
   &=
       \frac{1}{z}
       +\left(\frac{1}{6}-\frac{1}{3}\lambda \right)(2K)^{2}z
       +\left(\frac{7}{360}+\frac{1}{45}\lambda -\frac{1}{45}\lambda ^{2}\right)(2K)^{4}z^{3}
       +\cdots, \\
\label{eq:Laurent expansion of ns}
2K\ns(2Kz,k)
   &=
        \frac{1}{z}
        +\left(\frac{1}{6}+\frac{1}{6}\lambda \right)(2K)^{2}z
        +\left(\frac{7}{360}-\frac{11}{180}\lambda +\frac{7}{360}\lambda ^{2}\right)(2K)^{4}z^{3}
        +\cdots.
\end{align}
\url{http://functions.wolfram.com/EllipticFunctions/JacobiCS/06/01/01/}\\
\url{http://functions.wolfram.com/EllipticFunctions/JacobiDS/06/01/01/}\\
\url{http://functions.wolfram.com/EllipticFunctions/JacobiNS/06/01/01/}\\
\noindent
{\rm{(4)}}\;(Partial fraction expansions) (5.1) in \cite{W}
\begin{align}
\label{eq:Partial fraction cs}
2K\cs(2Kz,k)
   &=\sum_{m \in \mathbb{Z}}{}^{\!e}\sum_{n \in \mathbb{Z}}{}^{\!e}
\frac{(-1)^{m}}{m\tau +n+z}, \\
\label{eq:Partial fraction ds}
2K\ds(2Kz,k)
   &=\sum_{m \in \mathbb{Z}}{}^{\!e}\sum_{n \in \mathbb{Z}}{}^{\!e}
\frac{(-1)^{m+n}}{m\tau +n+z}, \\
\label{eq:Partial fraction ns}
2K\ns(2Kz,k)
   &=\sum_{m \in \mathbb{Z}}{}^{\!e}\sum_{n \in \mathbb{Z}}{}^{\!e}
\frac{(-1)^{n}}{m\tau +n+z}, 
\end{align}
where $\sum_{n \in \mathbb{Z}}{}^{\!e}$ is the Eisenstein convention 
$$
\sum_{n \in \mathbb{Z}}{}^{\!e}f(n):=f(0)+\sum_{n=1}^{\infty}\{f(n)+f(-n)\}.
$$
In particular, for any non zero constant $A$, we have
\begin{align}
\label{eq:residue cs}
\lim_{z \to -\frac{\mu \tau +\nu }{A}}
   \left(z+\frac{\mu \tau +\nu }{A}\right)2K\cs(2KAz,k)
   &=
   \frac{1}{A}(-1)^{\mu }, \\
\label{eq:residue ds}
\lim_{z \to -\frac{\mu \tau +\nu }{A}}
   \left(z+\frac{\mu \tau +\nu }{A}\right)2K\ds(2KAz,k)
   &=
   \frac{1}{A}(-1)^{\mu +\nu}, \\
\label{eq:residue ns}
\lim_{z \to -\frac{\mu \tau +\nu }{A}}
   \left(z+\frac{\mu \tau +\nu }{A}\right)2K\ns(2KAz,k)
   &=
   \frac{1}{A}(-1)^{\nu }.
\end{align}
\noindent
{\rm{(6)}}\;(Fourier expansions) p107 in \cite{W}
\begin{align}
\label{eq:Fourier cs}
2K\cs(2Kz,k)
   &=
   \pi \cot{(\pi z)}
   +
   \sum_{m=1}^{\infty}\frac{(-1)^{m}\pi \sin{(2\pi z)}}{\sin{(\pi (z+m\tau ))}\sin{(\pi (z-m\tau ))}}, \\
\label{eq:Fourier ds}
2K\ds(2Kz,k)
   &=
   \sum_{m\in \mathbb{Z}}\frac{\pi }{\sin{(\pi (z+m\tau ))}}, \\
\label{eq:Fourier ns}
2K\ns(2Kz,k)
   &=
   \sum_{m\in \mathbb{Z}}\frac{(-1)^{m}\pi }{\sin{(\pi (z+m\tau ))}}.
\end{align}
\noindent
{\rm{(4)}}\;(Derivations)
\begin{align}
\label{eq:deriv cs}
\cs^{\prime}(z,k)&=-\ds(z,k)\ns(z,k), \\
\label{eq:deriv ds}
\ds^{\prime}(z,k)&=-\cs(z,k)\ns(z,k), \\
\label{eq:deriv ns}
\ns^{\prime}(z,k)&=-\cs(z,k)\ds(z,k).    
\end{align}
\url{http://functions.wolfram.com/EllipticFunctions/JacobiCS/20/01/01/}\\
\url{http://functions.wolfram.com/EllipticFunctions/JacobiDS/20/01/01/}\\
\url{http://functions.wolfram.com/EllipticFunctions/JacobiNS/20/01/01/}\\
\noindent
{\rm{(5)}}\;(Relations between the Weierstrass $\wp$ function)
\begin{align}
\label{eq:wp and cs}
(2K\cs(2Kz,k))^{2}
   &=\wp(z,\tau)-\wp\left(\frac{1}{2},\tau \right), \\
\label{eq:wp and ds}
(2K\ds(2Kz,k))^{2}
   &=\wp(z,\tau)-\wp\left(\frac{1+\tau}{2},\tau \right), \\
\label{eq:wp and ns}
(2K\ns(2Kz,k))^{2}
   &=\wp(z,\tau)-\wp\left(\frac{\tau}{2},\tau \right), \\
\label{eq:wp and csdsns}
\wp^{\prime}(z,\tau)^{2}
   &=
   4(2K\cs(2Kz,k))(2K\ds(2Kz,k))(2K\ns(2Kz,k)).
\end{align}
Here, $\wp(z,\tau)$ is the Weierstrass $\wp$ function defined by
$$
\wp(z,\tau):=\frac{1}{z^{2}}+\sum_{\substack{m,n \in \mathbb{Z} \\ (m,n) \not=(0,0)}}\left\{\frac{1}{(m\tau +n+z)^{2}}-\frac{1}{{(m\tau +n)}^{2}}\right\}.
$$
\url{http://functions.wolfram.com/EllipticFunctions/JacobiCS/27/02/07/}\\
\url{http://functions.wolfram.com/EllipticFunctions/JacobiDS/27/02/07/}\\
\url{http://functions.wolfram.com/EllipticFunctions/JacobiNS/27/02/07/}\\
\noindent
{\rm{(5)}}\;(trigonometric degenerations)
\begin{align}
\label{eq:degeneration of k}
\lim_{\tau \to \sqrt{-1}\infty}
   k(\tau)&=0, \\
\label{eq:degeneration of K}
\lim_{\tau \to \sqrt{-1}\infty}
   2K(\tau)&=\pi, \\   
\label{eq:degeneration of cs}
\lim_{\tau \to \sqrt{-1}\infty}
   \cs(2K(z+w\tau),k)
   &=\begin{cases}
   (-1)^{w} \cot(\pi z) & (w \in \mathbb{Z}) \\
   -\sqrt{-1}(-1)^{\lfloor w\rfloor } & (w \not\in \mathbb{Z})
   \end{cases}, \\
\label{eq:degeneration of ds}
\lim_{\tau \to \sqrt{-1}\infty}
   \ds(2K(z+w\tau),k)
   &=\begin{cases}
   (-1)^{w} \csc(\pi z) & (w \in \mathbb{Z}) \\
   0 & (w \not\in \mathbb{Z})
   \end{cases}, \\
\label{eq:degeneration of ns}
\lim_{\tau \to \sqrt{-1}\infty}
   \ns(2K(z+w\tau),k)
   &=\begin{cases}
   \csc(\pi z) & (w \in \mathbb{Z}) \\
   0 & (w \not\in \mathbb{Z})
   \end{cases}.
\end{align}
Here, $\lfloor w \rfloor$ denotes the greatest integer not exceeding $w$. 
\end{lem}
For convenience, we put 
\begin{align}
f_{1,0}(z,\tau )&:=2K\cs(2Kz,k), \quad C_{1,0}(\tau ):=\left(-\frac{1}{3}+\frac{1}{6}\lambda \right)(2K)^{2}, \nonumber \\
f_{1,1}(z,\tau )&:=2K\ds(2Kz,k), \quad C_{1,1}(\tau ):=\left(\frac{1}{6}-\frac{1}{3}\lambda \right)(2K)^{2}, \nonumber \\ 
f_{0,1}(z,\tau )&:=2K\ns(2Kz,k), \quad C_{0,1}(\tau ):=\left(\frac{1}{6}+\frac{1}{6}\lambda \right)(2K)^{2}. \nonumber
\end{align}
According to these notations, we have the following expressions of 
parity (\ref{eq:parity of cs}) - (\ref{eq:parity of ns}), 
periodicity (\ref{eq:periodicity of cs}) - (\ref{eq:periodicity of ns}), 
Laurent expansions at $z=0$ (\ref{eq:Laurent expansion of cs}) - (\ref{eq:Laurent expansion of ns}), 
partial fraction expansions (\ref{eq:Partial fraction cs}) - (\ref{eq:Partial fraction ns}), 
residues at simple poles (\ref{eq:residue cs}) - (\ref{eq:residue ns}), 
derivations (\ref{eq:deriv cs}) - (\ref{eq:deriv ns}) 
and relations between the Weierstrass $\wp$ function (\ref{eq:wp and cs}) - (\ref{eq:wp and ns}) respectively. 
\begin{align}
\label{eq:fij parity}
f_{i,j}(-z,\tau )&=-f_{i,j}(z,\tau ), \\
\label{eq:fij periodicity}
f_{i,j}(z+\mu \tau +\nu ,\tau )
   &=
   (-1)^{i\mu +j\nu}f_{i,j}(z,\tau ), \\
\label{eq:fij Laurent z0}
f_{i,j}(z,\tau )
   &=
   \frac{1}{z}+C_{i,j}(\tau )z+O(z^{3}), \\
\label{eq:fij partial}
f_{i,j}(z,\tau )
   &=
   \sum_{m \in \mathbb{Z}}{}^{\!e}\sum_{n \in \mathbb{Z}}{}^{\!e}
   \frac{(-1)^{im+jn}}{m\tau +n+z}, \\   
\label{eq:fij residue}
\lim_{z \to -\frac{\mu \tau +\nu }{A}}
   \left(z+\frac{\mu \tau +\nu }{A}\right)f_{i,j}(Az,\tau )
   &=
   \frac{1}{A}(-1)^{i\mu +j\nu }, \\
\label{eq:fij derivation}
f_{i,j}^{\prime}(z,\tau )
   &=
   -f_{i+1,1}(z,\tau )f_{1,j+1}(z,\tau ), \\
\label{eq:fij wp}
f_{i,j}(z,\tau )^{2}
   &=
   \wp(z,\tau )-\wp\left(\frac{i+j\tau }{2},\tau \right).
\end{align}
Here indices of $f_{i,j}(z,\tau )$ are regarded as elements in $\mathbb{Z}/2\mathbb{Z}$. 
\begin{rmk}
{\rm{(1)}} Fukuhara-Yui use 
$$
\varphi (\tau, z):=\sqrt{\wp(z,\tau)-\wp\left(\frac{1}{2},\tau \right)}=\frac{1}{z}+O(z) \quad (z \to 0)
$$
instead of $2K\cs(2Kz,k)$. 
However, Fukuhara-Yui did not mention that $\varphi (\tau, z)$ is the Jacobi elliptic function $2K\cs(2Kz,k)$ exactly. \\
{\rm{(2)}} If we use Mumford's notations \cite{M}
$$
\theta _{0,0}(z,\tau ):=\theta _{3}(z,\tau ), \quad 
\theta _{1,0}(z,\tau ):=\theta _{2}(z,\tau ), \quad 
\theta _{0,1}(z,\tau ):=\theta _{4}(z,\tau ), \quad 
\theta _{1,1}(z,\tau ):=-\theta _{1}(z,\tau ), 
$$
then our $f_{i,j}(z,\tau )$ is written by 
$$
f_{i,j}(z,\tau )
   =
   -\frac{\pi }{2}\theta _{i,0}(0,\tau )\theta _{0,j}(0,\tau )\frac{\theta _{j+1,i+1}(z,\tau )}{\theta _{1,1}(z,\tau )}.
$$
\end{rmk}

\section{Main results}
Under the following we assume $a$ and $b$ are relatively prime positive numbers and $i,j,m,n \in \{0,1\}$. 
We mention and prove the main theorem.
\begin{thm}
\label{thm:main theorem1}
If $ia+mb$ or $ja+nb$ is odd, then 
\begin{align}
& \frac{1}{a}\!\!\!\!\!\sum_{\substack{\mu,\nu =0 \\ (\mu,\nu) \not=(0,0)}}^{a-1}\!\!\!\!\!(-1)^{i\mu +j\nu }
   f_{m,n}\left(b\frac{\mu \tau +\nu}{a},\tau \right)f_{ia+mb,ja+nb}\left(z+\frac{\mu \tau +\nu}{a},\tau \right) \nonumber \\
& \quad +\frac{1}{b}\!\!\!\!\!\sum_{\substack{\mu,\nu =0 \\ (\mu,\nu) \not=(0,0)}}^{b-1}\!\!\!\!\!(-1)^{m\mu +n\nu }
   f_{i,j}\left(a\frac{\mu \tau +\nu}{b},\tau \right)f_{ia+mb,ja+nb}\left(z+\frac{\mu \tau +\nu}{b},\tau \right) \nonumber \\
   \label{eq:main result}
   & \quad =
   -f_{i,j}(az,\tau )f_{m,n}(bz,\tau )
   +\frac{1}{ab}f_{ia+mb+1,1}(z,\tau )f_{1,ja+nb+1}(z,\tau ).
\end{align}
\end{thm}
\begin{proof}
We put 
\begin{align}
\Phi_{(i,j),(m,n)}((a,b),z,\tau)
   :=&   
   \frac{1}{a}\!\!\!\!\!\sum_{\substack{\mu,\nu =0 \\ (\mu,\nu) \not=(0,0)}}^{a-1}\!\!\!\!\!(-1)^{i\mu +j\nu }
   f_{m,n}\left(b\frac{\mu \tau +\nu}{a},\tau \right)f_{ia+mb,ja+nb}\left(z+\frac{\mu \tau +\nu}{a},\tau \right) \nonumber \\
   &\quad +\frac{1}{b}\!\!\!\!\!\sum_{\substack{\mu,\nu =0 \\ (\mu,\nu) \not=(0,0)}}^{b-1}\!\!\!\!\!(-1)^{m\mu +n\nu }
   f_{i,j}\left(a\frac{\mu \tau +\nu}{b},\tau \right)f_{ia+mb,ja+nb}\left(z+\frac{\mu \tau +\nu}{b},\tau \right), \nonumber \\
\Psi_{(i,j),(m,n)}((a,b),z,\tau)
   :=&
   -f_{i,j}(az,\tau )f_{m,n}(bz,\tau )+\frac{1}{ab}f_{ia+mb+1,1}(z,\tau )f_{1,ja+nb+1}(z,\tau ) \nonumber \\
   =&
   -f_{i,j}(az,\tau )f_{m,n}(bz,\tau )-\frac{1}{ab}f_{ia+mb,ja+nb}^{\prime}(z,\tau ). \nonumber    
\end{align}
and 
$$
U_{(i,j),(m,n)}((a,b),z,\tau)
   :=
   \Phi_{(i,j),(m,n)}((a,b),z,\tau)-\Psi_{(i,j),(m,n)}((a,b),z,\tau).
$$
Under the condition $2 \not|\,\,ia+mb$ or $2 \not|\,\,ja+nb$, we claim that 
$$
U_{(i,j),(m,n)}((a,b),z,\tau)\equiv 0.
$$

First we show that $\Phi_{(i,j),(m,n)}((a,b),z,\tau)$ and $\Psi_{(i,j),(m,n)}((a,b),z,\tau)$ have same periodicity. 
From periodicity (\ref{eq:fij periodicity}), for any integers $M$ and $N$ we have
\begin{align}
& \Psi_{(i,j),(m,n)}((a,b),z+M \tau +N ,\tau) \nonumber \\
   & \quad =
   -f_{i,j}(a(z+M \tau +N ),k)f_{m,n}(b(z+M \tau +N ),k) \nonumber \\
   & \quad \quad 
   +\frac{1}{ab}f_{ia+mb+1,1}(z+M \tau +N ,\tau )f_{1,ja+nb+1}(z+M \tau +N ,\tau ) \nonumber \\
   & \quad =
   -(-1)^{iaM +jaN}f_{i,j}(az,k)(-1)^{mbM +nbN}f_{m,n}(bz,k) \nonumber \\
   & \quad \quad 
   +\frac{1}{ab}(-1)^{(ia+mb+1)M +N}f_{ia+mb+1,1}(z,\tau )(-1)^{M +(ja+nb+1)N}f_{1,ja+nb+1}(z,\tau ) \nonumber \\
   & \quad =
   (-1)^{(ia+mb)M +(ja+nb)N }\Psi_{(i,j),(m,n)}((a,b),z,\tau). \nonumber
\end{align}
Similarly, for $\Phi_{(i,j),(m,n)}((a,b),z,\tau)$ we have 
$$
\Phi_{(i,j),(m,n)}((a,b),z+M \tau +N ,\tau)
   =
   (-1)^{(ia+mb)M +(ja+nb)N }\Phi_{(i,j),(m,n)}((a,b),z,\tau).
$$
Thus we obtain double periodicity of $U_{(i,j),(m,n)}((a,b),z,\tau)$
\begin{align}
\label{eq:double periodicity of U}
U_{(i,j),(m,n)}((a,b),z+M \tau +N ,\tau)
   =
   (-1)^{(ia+mb)M +(ja+nb)N}U_{(i,j),(m,n)}((a,b),z,\tau). 
\end{align}

Next we consider all the poles of $U_{(i,j),(m,n)}((a,b),z,\tau)$ and their Laurent expansions. 
We remark that $\Phi_{(i,j),(m,n)}((a,b),z,\tau)$, $\Psi_{(i,j),(m,n)}((a,b),z,\tau)$ and $U_{(i,j),(m,n)}((a,b),z,\tau)$ are holomorphic at $z=0$. 
Actually, since $a$ and $b$ are relatively prime positive integers, the Laurent expansions of $\Phi_{(i,j),(m,n)}((a,b),z,\tau)$ at $z=0$ is the following.  
\begin{align}
\Phi_{(i,j),(m,n)}((a,b),z,\tau)
   &=
   \frac{1}{a}\!\!\!\!\!\sum_{\substack{\mu,\nu =0 \\ (\mu,\nu) \not=(0,0)}}^{a-1}\!\!\!\!\!(-1)^{i\mu +j\nu }
   f_{m,n}\left(b\frac{\mu \tau +\nu}{a},\tau \right)f_{ia+mb,ja+nb}\left(\frac{\mu \tau +\nu}{a},\tau \right) \nonumber \\
   & \quad
   +\frac{1}{b}\!\!\!\!\!\sum_{\substack{\mu,\nu =0 \\ (\mu,\nu) \not=(0,0)}}^{b-1}\!\!\!\!\!(-1)^{m\mu +n\nu }
   f_{i,j}\left(a\frac{\mu \tau +\nu}{b},\tau \right)f_{ia+mb,ja+nb}\left(\frac{\mu \tau +\nu}{b},\tau \right)
   +O(z). \nonumber 
\end{align}
On the other hand, from (\ref{eq:fij derivation}) and (\ref{eq:fij Laurent z0}), we have 
\begin{align}
\Psi_{(i,j),(m,n)}((a,b),z,\tau)
   &=
   -f_{i,j}(az,\tau )f_{m,n}(bz,\tau )
   -\frac{1}{ab}f_{ia+mb,ja+nb}^{\prime}(z,\tau ) \nonumber \\
   &=
   -\left(\frac{1}{az}+C_{i,j}(\tau )az+O(z^{3}) \right)\left(\frac{1}{bz}+C_{m,n}(\tau )bz+O(z^{3}) \right) \nonumber \\
   & \quad 
   -\frac{1}{ab}\left(-\frac{1}{z^{2}}+C_{ia+mb,ja+nb}(\tau )+O(z^{2})\right) \nonumber \\
   &=
   -\frac{b}{a}C_{m,n}(\tau )-\frac{a}{b}C_{i,j}(\tau )-\frac{1}{ab}C_{ia+mb,ja+nb}(\tau )+O(z^{2}). \nonumber
\end{align}

Then we obtain the Laurent expansion of $U_{(i,j),(m,n)}((a,b),z,\tau)$ at $z=0$
\begin{align}
U_{(i,j),(m,n)}((a,b),z,\tau)
   &=
   -\frac{b}{a}C_{m,n}(\tau )-\frac{a}{b}C_{i,j}(\tau )-\frac{1}{ab}C_{ia+mb,ja+nb}(\tau ) \nonumber \\
   & \quad 
   -\frac{1}{a}\!\!\!\!\!\sum_{\substack{\mu,\nu =0 \\ (\mu,\nu) \not=(0,0)}}^{a-1}\!\!\!\!\!(-1)^{i\mu +j\nu }
   f_{m,n}\left(b\frac{\mu \tau +\nu}{a},\tau \right)f_{ia+mb,ja+nb}\left(\frac{\mu \tau +\nu}{a},\tau \right) \nonumber \\
   & \quad
   \label{eq:gen reciprocity}
   -\frac{1}{b}\!\!\!\!\!\sum_{\substack{\mu,\nu =0 \\ (\mu,\nu) \not=(0,0)}}^{b-1}\!\!\!\!\!(-1)^{m\mu +n\nu }
   f_{i,j}\left(a\frac{\mu \tau +\nu}{b},\tau \right)f_{ia+mb,ja+nb}\left(\frac{\mu \tau +\nu}{b},\tau \right)
   +O(z). 
\end{align}
Hence we investigate other poles. 
By the definition or partial fractional expansion (\ref{eq:fij partial}) of $f_{i,j}(z,\tau)$, all other poles of $\Phi_{(i,j),(m,n)}((a,b),z,\tau)$ and $\Psi_{(i,j),(m,n)}((a,b),z,\tau)$ are 
$$
-\frac{\mu \tau +\nu}{a}+M\tau +N, \quad \mu ,\nu =0,1,\ldots , a-1, \quad (\mu ,\nu )\not=(0,0), \quad M,N \in \mathbb{Z}
$$
or
$$
-\frac{\mu \tau +\nu}{b}+M\tau +N, \quad \mu ,\nu =0,1,\ldots , b-1, \quad (\mu ,\nu )\not=(0,0), \quad M,N \in \mathbb{Z}. 
$$
Furthermore, all the poles are simple and the residues at these poles are equal. 
Actually, from (\ref{eq:fij residue}), we have 
\begin{align}
& \lim_{z \to -\frac{\mu \tau +\nu }{a}-M\tau -N}
   \left(z+\frac{\mu \tau +\nu }{a}+M\tau +N\right)\Phi_{(i,j),(m,n)}((a,b),z,\tau) \nonumber \\
   & \quad =
   -(-1)^{iaM +jaN}f_{m,n}\left(-b\frac{\mu \tau +\nu}{a}-bM\tau -bN,\tau \right) \nonumber \\
   & \quad \quad \cdot 
   \lim_{z +M\tau +N\to -\frac{\mu \tau +\nu }{a}}
   \left(z+M\tau +N+\frac{\mu \tau +\nu }{a}\right)
   f_{i,j}(a(z+M\tau +N),\tau ) \nonumber \\
   & \quad =
   (-1)^{iaM +jaN}(-1)^{mbM+nbN}f_{m,n}\left(b\frac{\mu \tau +\nu}{a},\tau \right)\frac{(-1)^{i\mu +j\nu}}{a} \nonumber \\
   & \quad =
   \frac{1}{a}(-1)^{i\mu +j\nu}(-1)^{(ia+mb)M +(ja+nb)N}f_{m,n}\left(b\frac{\mu \tau +\nu}{a},\tau \right) \nonumber
\end{align}
and 
\begin{align}
& \lim_{z \to -\frac{\mu \tau +\nu }{a}-M\tau -N}
   \left(z+\frac{\mu \tau +\nu }{a}+M\tau +N\right)\Psi_{(i,j),(m,n)}((a,b),z,\tau) \nonumber \\
   & \quad =
   \frac{1}{a}(-1)^{i\mu +j\nu }(-1)^{(ia+mb)M+(ja+nb)N}
   f_{m,n}\left(b\frac{\mu \tau +\nu}{a},\tau \right) \nonumber \\
   & \quad \quad \cdot 
   \lim_{z +M\tau +N+\frac{\mu \tau +\nu }{a}\to 0}
   \left(z +M\tau +N+\frac{\mu \tau +\nu }{a}\right)
   f_{ia+mb,ja+nb}\left(z+M\tau +N+\frac{\mu \tau +\nu}{a},\tau \right) \nonumber \\
   & \quad =
   \frac{1}{a}(-1)^{i\mu +j\nu }(-1)^{(ia+mb)M+(ja+nb)N}
   f_{m,n}\left(b\frac{\mu \tau +\nu}{a},\tau \right). \nonumber
\end{align}
Thus for $M, N \in \mathbb{Z}$, $\mu, \nu \in \{0,1,\ldots , a-1\}$ and $(\mu ,\nu )\not=(0,0)$, 
$$
\lim_{z \to -\frac{\mu \tau +\nu }{a}-M\tau -N}
   \left(z+M\tau +N+\frac{\mu \tau +\nu }{a}\right)U_{(i,j),(m,n)}((a,b),z,\tau)
   =
   0.
$$
Similarly, for $M, N \in \mathbb{Z}$, $\mu, \nu \in \{0,1,\ldots , b-1\}$ and $(\mu ,\nu )\not=(0,0)$ we have 
$$
\lim_{z \to -\frac{\mu \tau +\nu }{b}-M\tau -N}
   \left(z+M\tau +N+\frac{\mu \tau +\nu }{b}\right)U_{(i,j),(m,n)}((a,b),z,\tau)
   =
   0. 
$$
Therefore $U_{(i,j),(m,n)}((a,b),z,\tau)$ is an entire function.

Summarizing the above discussion, $U_{(i,j),(m,n)}((a,b),z,\tau)$ is a doubly periodic entire function on $\mathbb{C}$. 
Then by the well-known Liouville theorem, there exists a constant $c_{(i,j),(m,n)}((a,b),\tau )$ such that 
$$
U_{(i,j),(m,n)}((a,b),z,\tau)
   =
   c_{(i,j),(m,n)}((a,b),\tau ). 
$$

If $ia+mb$ is odd, changing the variable from $z$ to $z+\tau$, we have
\begin{align}
c_{(i,j),(m,n)}((a,b),\tau )
   &=
   U_{(i,j),(m,n)}((a,b),z+\tau ,\tau) \nonumber \\
   &=
   (-1)^{ia+mb}U_{(i,j),(m,n)}((a,b),z ,\tau) \nonumber \\
   &=
   -U_{(i,j),(m,n)}((a,b),z,\tau) \nonumber \\
   &=
   -c_{(i,j),(m,n)}((a,b),\tau ). \nonumber
\end{align}
Here the second equality follows from double periodicity (\ref{eq:double periodicity of U}). 
If $ia+mb$ is even, from the assumption of theorem, then $ja+nb$ is odd. 
Hence, changing the variable from $z$ to $z+1$, we have 
\begin{align}
c_{(i,j),(m,n)}((a,b),\tau )
   &=
   U_{(i,j),(m,n)}((a,b),z+1 ,\tau) \nonumber \\
   &=
   (-1)^{ja+nb}U_{(i,j),(m,n)}((a,b),z ,\tau) \nonumber \\
   &=
   -U_{(i,j),(m,n)}((a,b),z ,\tau) \nonumber \\
   &=
   -c_{(i,j),(m,n)}((a,b),\tau ). \nonumber
\end{align}
Therefore $c_{(i,j),(m,n)}((a,b),\tau )\equiv 0$ in any cases and we obtain the conclusion. 
\end{proof}
Expanding both side of (\ref{eq:main result}) and comparing coefficients of $z^{2N}$, we obtain reciprocity laws for elliptic Dedekind sums, which is a natural generalization of Fukuhara-Yui's main result (Theorem\,2.2 (1) in \cite{FY}). 
\begin{thm}[Reciprocity laws for elliptic Dedekind sums]
\label{thm:main theorem2}
If $ia+mb$ or $ja+nb$ is odd, then 
\begin{align}
& \frac{1}{(2N)!}\frac{1}{a}\!\!\!\!\!\sum_{\substack{\mu,\nu =0 \\ (\mu,\nu) \not=(0,0)}}^{a-1}\!\!\!\!\!(-1)^{i\mu +j\nu }
   f_{m,n}\left(b\frac{\mu \tau +\nu}{a},\tau \right)f_{ia+mb,ja+nb}^{(2N)}\left(\frac{\mu \tau +\nu}{a},\tau \right) \nonumber \\
   & \quad 
   +\frac{1}{(2N)!}\frac{1}{b}\!\!\!\!\!\sum_{\substack{\mu,\nu =0 \\ (\mu,\nu) \not=(0,0)}}^{b-1}\!\!\!\!\!(-1)^{m\mu +n\nu }
   f_{i,j}\left(a\frac{\mu \tau +\nu}{b},\tau \right)f_{ia+mb,ja+nb}^{(2N)}\left(\frac{\mu \tau +\nu}{b},\tau \right) \nonumber \\
\label{eq:main theorem2}
   & \quad =
   -\sum_{s=0}^{N}C_{i,j}(s,\tau )C_{m,n}(N-s,\tau )a^{2s-1}b^{2N-2s-1}
   -\frac{1}{ab}(2N+1)C_{ia+mb,ja+nb}(N, \tau ), 
\end{align}
where $f_{ia+mb,ja+nb}^{(2N)}(z)$ is the $2N$-th derivative of the $f_{ia+mb,ja+nb}(z)$, 
and $C_{i,j}(s,\tau)$ is the coefficients of the Laurent expansions of $f_{i,j}(z,\tau )$ at $z=0$ 
$$
f_{ij}(z,\tau)
   =\frac{1}{z}\sum_{s=0}^{\infty}C_{i,j}(s,\tau)z^{2s}, \quad C_{i,j}(0,\tau)=1, \quad C_{i,j}(1,\tau):=C_{i,j}(\tau ). 
$$
\end{thm}
In particular, considering the case of $N=0$ of (\ref{eq:main theorem2}) or taking the limit $z \to 0$ in (\ref{eq:gen reciprocity}), we obtain reciprocity laws for elliptic Dedekind sums. 
\begin{cor}
\label{thm:main theorem3}
If $ia+mb$ or $ja+nb$ is odd, then 
\begin{align}
& \frac{1}{a}\!\!\!\!\!\sum_{\substack{\mu,\nu =0 \\ (\mu,\nu) \not=(0,0)}}^{a-1}\!\!\!\!\!(-1)^{i\mu +j\nu }
   f_{m,n}\left(b\frac{\mu \tau +\nu}{a},\tau \right)f_{ia+mb,ja+nb}\left(\frac{\mu \tau +\nu}{a},\tau \right) \nonumber \\
   & \quad 
   +\frac{1}{b}\!\!\!\!\!\sum_{\substack{\mu,\nu =0 \\ (\mu,\nu) \not=(0,0)}}^{b-1}\!\!\!\!\!(-1)^{m\mu +n\nu }
   f_{i,j}\left(a\frac{\mu \tau +\nu}{b},\tau \right)f_{ia+mb,ja+nb}\left(\frac{\mu \tau +\nu}{b},\tau \right) \nonumber \\
\label{eq:main result2}
   & \quad =
   -\frac{b}{a}C_{m,n}(\tau )-\frac{a}{b}C_{i,j}(\tau )-\frac{1}{ab}C_{ia+mb,ja+nb}(\tau ).
\end{align}
\end{cor}

\section{All the examples of (\ref{eq:main result}) and (\ref{eq:main result2})}
In this section, we give all the examples of (\ref{eq:main result}) and (\ref{eq:main result2}) up to the constant factor $(2K)^{2}$ explicitly. 

\subsection{$(i,j)=(1,0)$, $(m,n)=(1,0)$}
\subsubsection{$2\not|\,\,a+b$}
In this case, (\ref{eq:main result}) and (\ref{eq:main result2}) are Fukuhara-Yui's results (\ref{eq:elliptic prot formula}) and (\ref{eq:elliptic prot formula2}) respectively. 

\subsection{$(i,j)=(1,1)$, $(m,n)=(1,0)$}

\subsubsection{$2\not|\,\,a+b$, $2\not|\,\,a$}
\begin{align}
& \frac{1}{a}\!\!\!\!\!\sum_{\substack{\mu,\nu =0 \\ (\mu,\nu) \not=(0,0)}}^{a-1}\!\!\!\!\!(-1)^{\mu +\nu }
   \cs\left(2Kb\frac{\mu \tau +\nu}{a},k \right)\ds\left(2K\left(z+\frac{\mu \tau +\nu}{a}\right),k \right) \nonumber \\
& \quad +
   \frac{1}{b}\!\!\!\!\!\sum_{\substack{\mu,\nu =0 \\ (\mu,\nu) \not=(0,0)}}^{b-1}\!\!\!\!\!(-1)^{\mu }
   \ds\left(2Ka\frac{\mu \tau +\nu}{b},k \right)\ds\left(2K\left(z+\frac{\mu \tau +\nu}{b}\right),k \right) \nonumber \\
\label{eq:421-1}
   & \quad =
   -\ds(2Kaz,k )\cs(2Kbz,k )
   +\frac{1}{ab}\ns(2Kz,k )\cs(2Kz,k ), \\
& \frac{1}{a}\!\!\!\!\!\sum_{\substack{\mu,\nu =0 \\ (\mu,\nu) \not=(0,0)}}^{a-1}\!\!\!\!\!(-1)^{\mu +\nu }
   \cs\left(2Kb\frac{\mu \tau +\nu}{a},k \right)\ds\left(2K\frac{\mu \tau +\nu}{a},k \right) \nonumber \\
& \quad +
   \frac{1}{b}\!\!\!\!\!\sum_{\substack{\mu,\nu =0 \\ (\mu,\nu) \not=(0,0)}}^{b-1}\!\!\!\!\!(-1)^{\mu }
   \ds\left(2Ka\frac{\mu \tau +\nu}{b},k \right)\ds\left(2K\frac{\mu \tau +\nu}{b},k \right) \nonumber \\
\label{eq:421-2}
   & \quad =
   \frac{-a^{2}+2b^{2}-1}{6ab}+\frac{2a^{2}-b^{2}+2}{6ab}\lambda (\tau).
\end{align}
By taking the limit $\tau \to \sqrt{-1}\infty$ and (\ref{eq:degeneration of k}) - (\ref{eq:degeneration of ns}), (\ref{eq:421-1}) and (\ref{eq:421-2}) degenerate to (\ref{eq:another prot thm1.4}) and (\ref{eq:2another prot thm1.4}) respectively. 

\subsubsection{$2\not|\,\,a+b$, $2\,\,|\,\,a$}
\begin{align}
& \frac{1}{a}\!\!\!\!\!\sum_{\substack{\mu,\nu =0 \\ (\mu,\nu) \not=(0,0)}}^{a-1}\!\!\!\!\!(-1)^{\mu +\nu }
   \cs\left(2Kb\frac{\mu \tau +\nu}{a},k \right)
   \cs\left(2K\left(z+\frac{\mu \tau +\nu}{a}\right),k \right) \nonumber \\
& \quad +\frac{1}{b}\!\!\!\!\!\sum_{\substack{\mu,\nu =0 \\ (\mu,\nu) \not=(0,0)}}^{b-1}\!\!\!\!\!(-1)^{\mu }
   \ds\left(2Ka\frac{\mu \tau +\nu}{b},k \right)
   \cs\left(2K\left(z+\frac{\mu \tau +\nu}{b}\right),k \right) \nonumber \\
\label{eq:422-1}
   & \quad =
   -\ds(2Kaz,k )\cs(2Kbz,k )
   +\frac{1}{ab}\ns(2Kz,k )\ds(2Kz,k ), 
\end{align}
\begin{align}
& \frac{1}{a}\!\!\!\!\!\sum_{\substack{\mu,\nu =0 \\ (\mu,\nu) \not=(0,0)}}^{a-1}\!\!\!\!\!(-1)^{\mu +\nu }
   \cs\left(2Kb\frac{\mu \tau +\nu}{a},k \right)
   \cs\left(2K\frac{\mu \tau +\nu}{a},k \right) \nonumber \\
& \quad +\frac{1}{b}\!\!\!\!\!\sum_{\substack{\mu,\nu =0 \\ (\mu,\nu) \not=(0,0)}}^{b-1}\!\!\!\!\!(-1)^{\mu }
   \ds\left(2Ka\frac{\mu \tau +\nu}{b},k \right)
   \cs\left(2K\frac{\mu \tau +\nu}{b},k \right) \nonumber \\
\label{eq:422-2}
   & \quad =
   \frac{-a^{2}+2b^{2}+2}{6ab}+\frac{2a^{2}-b^{2}-1}{6ab}\lambda (\tau).
\end{align}
By taking the limit $\tau \to \sqrt{-1}\infty$ and (\ref{eq:degeneration of k}) - (\ref{eq:degeneration of ns}), we have
\begin{align}
& \lim_{\tau \to \sqrt{-1}\infty}
   \Phi_{(1,1),(1,0)}((a,b),z,\tau) \nonumber \\
   & \quad =
   \lim_{\tau \to \sqrt{-1}\infty}
   \frac{1}{a}\sum_{\nu =1}^{a-1}(-1)^{\nu }
   \cs\left(2Kb\frac{\nu}{a},k \right)
   \cs\left(2K\left(z+\frac{\nu}{a}\right),k \right) \nonumber \\
      & \quad \quad +
   \lim_{\tau \to \sqrt{-1}\infty}
   \frac{1}{a}\sum_{\mu =1}^{a-1}\sum_{\nu =0}^{a-1}(-1)^{\mu +\nu }
   \cs\left(2Kb\frac{\mu \tau +\nu}{a},k \right)
   \cs\left(2K\left(z+\frac{\mu \tau +\nu}{a}\right),k \right) \nonumber \\
      & \quad \quad +
   \lim_{\tau \to \sqrt{-1}\infty}
   \frac{1}{b}\!\!\!\!\!\sum_{\substack{\mu,\nu =0 \\ (\mu,\nu) \not=(0,0)}}^{b-1}\!\!\!\!\!(-1)^{\mu }
   \ds\left(2Ka\frac{\mu \tau +\nu}{b},k \right)
   \cs\left(2K\left(z+\frac{\mu \tau +\nu}{b}\right),k \right) \nonumber \\
   & \quad =
   \frac{1}{a}\sum_{\nu =1}^{a-1}(-1)^{\nu}\cot{\left(\frac{\pi b \nu }{a}\right)}\cot{\left(\pi \left(z+\frac{\nu}{a}\right)\right)} 
   +\frac{1}{b}\sum_{\nu =1}^{b-1}\csc{\left(\frac{\pi a \nu}{b}\right)}\cot{\left(\pi \left(z+\frac{\nu}{b}\right)\right)} \nonumber \\
\label{eq:last equality}
      & \quad \quad +
   \frac{1}{a}\sum_{\mu =1}^{a-1}(-1)^{\left\lfloor \frac{b\mu }{a}\right\rfloor +\mu -1}\sum_{\nu =0}^{a-1}(-1)^{\nu }, \\
& \lim_{\tau \to \sqrt{-1}\infty}
   \Psi_{(1,1),(1,0)}((a,b),z,\tau) \nonumber \\
   & \quad =
   -\csc{(\pi az)}\cot{(\pi bz)}+\frac{1}{ab}\csc(\pi z)^{2}. \nonumber
\end{align}
Since $a$ is even and the third term in (\ref{eq:last equality}) vanishes, (\ref{eq:422-1}) and (\ref{eq:422-2}) degenerate to (\ref{eq:another prot thm1.3}) and (\ref{eq:2another prot thm1.3}) respectively.

\subsubsection{$2\,\,|\,\,a+b$, $2\not|\,\,a$}
\begin{align}
& \frac{1}{a}\!\!\!\!\!\sum_{\substack{\mu,\nu =0 \\ (\mu,\nu) \not=(0,0)}}^{a-1}\!\!\!\!\!(-1)^{\mu +\nu }
   \cs\left(2Kb\frac{\mu \tau +\nu}{a},k \right)
   \ns\left(2K\left(z+\frac{\mu \tau +\nu}{a}\right),k \right) \nonumber \\
& \quad +\frac{1}{b}\!\!\!\!\!\sum_{\substack{\mu,\nu =0 \\ (\mu,\nu) \not=(0,0)}}^{b-1}\!\!\!\!\!(-1)^{\mu }
   \ds\left(2Ka\frac{\mu \tau +\nu}{b},k \right)
   \ns\left(2K\left(z+\frac{\mu \tau +\nu}{b}\right),k \right) \nonumber \\
\label{eq:423-1}
   & \quad =
   -\ds(2Kaz,k )\cs(2Kbz,k )
   +\frac{1}{ab}\ds(2Kz,k )\cs(2Kz,k ), \\
& \frac{1}{a}\!\!\!\!\!\sum_{\substack{\mu,\nu =0 \\ (\mu,\nu) \not=(0,0)}}^{a-1}\!\!\!\!\!(-1)^{\mu +\nu }
   \cs\left(2Kb\frac{\mu \tau +\nu}{a},k \right)
   \ns\left(2K\frac{\mu \tau +\nu}{a},k \right) \nonumber \\
& \quad +\frac{1}{b}\!\!\!\!\!\sum_{\substack{\mu,\nu =0 \\ (\mu,\nu) \not=(0,0)}}^{b-1}\!\!\!\!\!(-1)^{\mu }
   \ds\left(2Ka\frac{\mu \tau +\nu}{b},k \right)
   \ns\left(2K\frac{\mu \tau +\nu}{b},k \right) \nonumber \\
\label{eq:423-2}
   & \quad =
   \frac{-a^{2}+2b^{2}-1}{6ab}+\frac{2a^{2}-b^{2}-1}{6ab}\lambda (\tau).
\end{align}
Taking the limit $\tau \to \sqrt{-1}\infty$, (\ref{eq:423-1}) and (\ref{eq:423-2}) degenerate to (\ref{eq:another prot thm1.4}) and (\ref{eq:2another prot thm1.4}) respectively.

\subsection{$(i,j)=(0,1)$, $(m,n)=(1,0)$}

\subsubsection{$2\not|\,\,b$, $2\not|\,\,a$}
\begin{align}
& \frac{1}{a}\!\!\!\!\!\sum_{\substack{\mu,\nu =0 \\ (\mu,\nu) \not=(0,0)}}^{a-1}\!\!\!\!\!(-1)^{\nu }
   \cs\left(2Kb\frac{\mu \tau +\nu}{a},k \right)
   \ds\left(2K\left(z+\frac{\mu \tau +\nu}{a}\right),k \right) \nonumber \\
& \quad +\frac{1}{b}\!\!\!\!\!\sum_{\substack{\mu,\nu =0 \\ (\mu,\nu) \not=(0,0)}}^{b-1}\!\!\!\!\!(-1)^{\mu }
   \ns\left(2Ka\frac{\mu \tau +\nu}{b},k \right)
   \ds\left(2K\left(z+\frac{\mu \tau +\nu}{b}\right),k \right) \nonumber \\
\label{eq:431-1}
   & \quad =
   -\ns(2Kaz,k )\cs(2Kbz,k )
   +\frac{1}{ab}\ns(2Kz,k )\cs(z,k ), 
\end{align}
\begin{align}
& \frac{1}{a}\!\!\!\!\!\sum_{\substack{\mu,\nu =0 \\ (\mu,\nu) \not=(0,0)}}^{a-1}\!\!\!\!\!(-1)^{\nu }
   \cs\left(2Kb\frac{\mu \tau +\nu}{a},k \right)
   \ds\left(2K\frac{\mu \tau +\nu}{a},k \right) \nonumber \\
& \quad +\frac{1}{b}\!\!\!\!\!\sum_{\substack{\mu,\nu =0 \\ (\mu,\nu) \not=(0,0)}}^{b-1}\!\!\!\!\!(-1)^{\mu }
   \ns\left(2Ka\frac{\mu \tau +\nu}{b},k \right)
   \ds\left(2K\frac{\mu \tau +\nu}{b},k \right) \nonumber \\
\label{eq:431-2}
   & \quad =
   \frac{-a^{2}+2b^{2}-1}{6ab}+\frac{-a^{2}-b^{2}+2}{6ab}\lambda (\tau).
\end{align}
Taking the limit $\tau \to \sqrt{-1}\infty$, (\ref{eq:431-1}) and (\ref{eq:431-2}) degenerate to (\ref{eq:another prot thm1.4}) and (\ref{eq:2another prot thm1.4}) respectively.

\subsubsection{$2\not|\,\,b$, $2\,\,|\,\,a$}
\begin{align}
& \frac{1}{a}\!\!\!\!\!\sum_{\substack{\mu,\nu =0 \\ (\mu,\nu) \not=(0,0)}}^{a-1}\!\!\!\!\!(-1)^{\nu }
   \cs\left(2Kb\frac{\mu \tau +\nu}{a},k \right)
   \cs\left(2K\left(z+\frac{\mu \tau +\nu}{a}\right),k \right) \nonumber \\
& \quad 
   +\frac{1}{b}\!\!\!\!\!\sum_{\substack{\mu,\nu =0 \\ (\mu,\nu) \not=(0,0)}}^{b-1}\!\!\!\!\!(-1)^{\mu }
   \ns\left(2Ka\frac{\mu \tau +\nu}{b},k \right)
   \cs\left(2K\left(z+\frac{\mu \tau +\nu}{b}\right),k \right) \nonumber \\
\label{eq:432-1}
   & \quad =
   -\ns(2Kaz,k )\cs(2Kbz,k )
   +\frac{1}{ab}\ns(2Kz,k )\ds(2Kz,k ), \\
& \frac{1}{a}\!\!\!\!\!\sum_{\substack{\mu,\nu =0 \\ (\mu,\nu) \not=(0,0)}}^{a-1}\!\!\!\!\!(-1)^{\nu }
   \cs\left(2Kb\frac{\mu \tau +\nu}{a},k \right)
   \cs\left(2K\frac{\mu \tau +\nu}{a},k \right) \nonumber \\
& \quad 
   +\frac{1}{b}\!\!\!\!\!\sum_{\substack{\mu,\nu =0 \\ (\mu,\nu) \not=(0,0)}}^{b-1}\!\!\!\!\!(-1)^{\mu }
   \ns\left(2Ka\frac{\mu \tau +\nu}{b},k \right)
   \cs\left(2K\frac{\mu \tau +\nu}{b},k \right) \nonumber \\
\label{eq:432-2}
   & \quad =
   \frac{-a^{2}+2b^{2}+2}{6ab}+\frac{-a^{2}-b^{2}-1}{6ab}\lambda (\tau).
\end{align}
Taking the limit $\tau \to \sqrt{-1}\infty$, (\ref{eq:432-1}) and (\ref{eq:432-2}) degenerate to (\ref{eq:another prot thm1.3}) and (\ref{eq:2another prot thm1.3}) respectively.

\subsubsection{$2\,\,|\,\,b$, $2\not|\,\,a$}
\begin{align}
& \frac{1}{a}\!\!\!\!\!\sum_{\substack{\mu,\nu =0 \\ (\mu,\nu) \not=(0,0)}}^{a-1}\!\!\!\!\!(-1)^{\nu }
   \cs\left(2Kb\frac{\mu \tau +\nu}{a},k \right)
   \ns\left(2K\left(z+\frac{\mu \tau +\nu}{a}\right),k \right) \nonumber \\
& \quad +\frac{1}{b}\!\!\!\!\!\sum_{\substack{\mu,\nu =0 \\ (\mu,\nu) \not=(0,0)}}^{b-1}\!\!\!\!\!(-1)^{\mu }
   \ns\left(2Ka\frac{\mu \tau +\nu}{b},k \right)
   \ns\left(2K\left(z+\frac{\mu \tau +\nu}{b}\right),k \right) \nonumber \\
\label{eq:433-1}
   & \quad =
   -\ns(2Kaz,k )cs(2Kbz,k )
   +\frac{1}{ab}\ds(2Kz,k )\cs(2Kz,k ), \\
& \frac{1}{a}\!\!\!\!\!\sum_{\substack{\mu,\nu =0 \\ (\mu,\nu) \not=(0,0)}}^{a-1}\!\!\!\!\!(-1)^{\nu }
   \cs\left(2Kb\frac{\mu \tau +\nu}{a},k \right)
   \ns\left(2K\frac{\mu \tau +\nu}{a},k \right) \nonumber \\
& \quad +\frac{1}{b}\!\!\!\!\!\sum_{\substack{\mu,\nu =0 \\ (\mu,\nu) \not=(0,0)}}^{b-1}\!\!\!\!\!(-1)^{\mu }
   \ns\left(2Ka\frac{\mu \tau +\nu}{b},k \right)
   \ns\left(2K\frac{\mu \tau +\nu}{b},k \right) \nonumber \\
\label{eq:433-2}
   & \quad =
   \frac{-a^{2}+2b^{2}-1}{6ab}+\frac{-a^{2}-b^{2}-1}{6ab}\lambda (\tau).
\end{align}
Taking the limit $\tau \to \sqrt{-1}\infty$, (\ref{eq:433-1}) and (\ref{eq:433-2}) degenerate to (\ref{eq:another prot thm1.4}) and (\ref{eq:2another prot thm1.4}) respectively.

\subsection{$(i,j)=(1,1)$, $(m,n)=(1,1)$}

\subsubsection{$2\not|\,\,a+b$}
\begin{align}
& \frac{1}{a}\!\!\!\!\!\sum_{\substack{\mu,\nu =0 \\ (\mu,\nu) \not=(0,0)}}^{a-1}\!\!\!\!\!(-1)^{\mu +\nu }
   \ds\left(2Kb\frac{\mu \tau +\nu}{a},k \right)
   \ds\left(2K\left(z+\frac{\mu \tau +\nu}{a}\right),k \right) \nonumber \\
& \quad +\frac{1}{b}\!\!\!\!\!\sum_{\substack{\mu,\nu =0 \\ (\mu,\nu) \not=(0,0)}}^{b-1}\!\!\!\!\!(-1)^{\mu +\nu }
   \ds\left(2Ka\frac{\mu \tau +\nu}{b},k \right)
   \ds\left(2K\left(z+\frac{\mu \tau +\nu}{b}\right),k \right) \nonumber \\
\label{eq:441-1}
   & \quad =
   -\ds(2Kaz,k )\ds(2Kbz,k )
   +\frac{1}{ab}\ns(2Kz,k )\cs(2Kz,k ), 
\end{align}
\begin{align}
& \frac{1}{a}\!\!\!\!\!\sum_{\substack{\mu,\nu =0 \\ (\mu,\nu) \not=(0,0)}}^{a-1}\!\!\!\!\!(-1)^{\mu +\nu }
   \ds\left(2Kb\frac{\mu \tau +\nu}{a},k \right)
   \ds\left(2K\frac{\mu \tau +\nu}{a},k \right) \nonumber \\
& \quad +\frac{1}{b}\!\!\!\!\!\sum_{\substack{\mu,\nu =0 \\ (\mu,\nu) \not=(0,0)}}^{b-1}\!\!\!\!\!(-1)^{\mu +\nu }
   \ds\left(2Ka\frac{\mu \tau +\nu}{b},k \right)
   \ds\left(2K\frac{\mu \tau +\nu}{b},k \right) \nonumber \\
\label{eq:441-2}
   & \quad =
   -\frac{a^{2}+b^{2}+1}{6ab}\left(1-2\lambda (\tau)\right). 
\end{align}
Taking the limit $\tau \to \sqrt{-1}\infty$, (\ref{eq:441-1}) and (\ref{eq:441-2}) degenerate to (\ref{eq:another prot thm1.2}) and (\ref{eq:2another prot thm1.2}) respectively.

\subsection{$(i,j)=(0,1)$, $(m,n)=(1,1)$}

\subsubsection{$2\not|\,\,b$, $2\not|\,\,a+b$}
\begin{align}
& \frac{1}{a}\!\!\!\!\!\sum_{\substack{\mu,\nu =0 \\ (\mu,\nu) \not=(0,0)}}^{a-1}\!\!\!\!\!(-1)^{\nu }
   \ds\left(2Kb\frac{\mu \tau +\nu}{a},k \right)
   \ds\left(2K\left(z+\frac{\mu \tau +\nu}{a}\right),k \right) \nonumber \\
& \quad +\frac{1}{b}\!\!\!\!\!\sum_{\substack{\mu,\nu =0 \\ (\mu,\nu) \not=(0,0)}}^{b-1}\!\!\!\!\!(-1)^{\mu +\nu }
   \ns\left(2Ka\frac{\mu \tau +\nu}{b},k \right)
   \ds\left(2K\left(z+\frac{\mu \tau +\nu}{b}\right),k \right) \nonumber \\
\label{eq:451-1}
   & \quad =
   -\ns(2Kaz,k )\ds(2Kbz,k )
   +\frac{1}{ab}\ns(2Kz,k )\cs(2Kz,k ), \\
& \frac{1}{a}\!\!\!\!\!\sum_{\substack{\mu,\nu =0 \\ (\mu,\nu) \not=(0,0)}}^{a-1}\!\!\!\!\!(-1)^{\nu }
   \ds\left(2Kb\frac{\mu \tau +\nu}{a},k \right)
   \ds\left(2K\frac{\mu \tau +\nu}{a},k \right) \nonumber \\
& \quad +\frac{1}{b}\!\!\!\!\!\sum_{\substack{\mu,\nu =0 \\ (\mu,\nu) \not=(0,0)}}^{b-1}\!\!\!\!\!(-1)^{\mu +\nu }
   \ns\left(2Ka\frac{\mu \tau +\nu}{b},k \right)
   \ds\left(2K\frac{\mu \tau +\nu}{b},k \right) \nonumber \\
\label{eq:451-2}
   & \quad =
   \frac{-a^{2}-b^{2}-1}{6ab}+\frac{-a^{2}+2b^{2}+2}{6ab}\lambda (\tau).
\end{align}
Taking the limit $\tau \to \sqrt{-1}\infty$, (\ref{eq:451-1}) and (\ref{eq:451-2}) degenerate to (\ref{eq:another prot thm1.2}) and (\ref{eq:2another prot thm1.2}) respectively.

\subsubsection{$2\not|\,\,b$, $2\,\,|\,\,a+b$}
\begin{align}
& \frac{1}{a}\!\!\!\!\!\sum_{\substack{\mu,\nu =0 \\ (\mu,\nu) \not=(0,0)}}^{a-1}\!\!\!\!\!(-1)^{\nu }
   \ds\left(2Kb\frac{\mu \tau +\nu}{a},k \right)
   \cs\left(2K\left(z+\frac{\mu \tau +\nu}{a}\right),k \right) \nonumber \\
& \quad +\frac{1}{b}\!\!\!\!\!\sum_{\substack{\mu,\nu =0 \\ (\mu,\nu) \not=(0,0)}}^{b-1}\!\!\!\!\!(-1)^{\mu +\nu }
   \ns\left(2Ka\frac{\mu \tau +\nu}{b},k \right)
   \cs\left(2K\left(z+\frac{\mu \tau +\nu}{b}\right),k \right) \nonumber \\
\label{eq:452-1}
   & \quad =
   -\ns(2Kaz,k )\ds(2Kbz,k )
   +\frac{1}{ab}\ns(2Kz,k )\ds(2Kz,k ), \\
& \frac{1}{a}\!\!\!\!\!\sum_{\substack{\mu,\nu =0 \\ (\mu,\nu) \not=(0,0)}}^{a-1}\!\!\!\!\!(-1)^{\nu }
   \ds\left(2Kb\frac{\mu \tau +\nu}{a},k \right)
   \cs\left(2K\frac{\mu \tau +\nu}{a},k \right) \nonumber \\
& \quad +\frac{1}{b}\!\!\!\!\!\sum_{\substack{\mu,\nu =0 \\ (\mu,\nu) \not=(0,0)}}^{b-1}\!\!\!\!\!(-1)^{\mu +\nu }
   \ns\left(2Ka\frac{\mu \tau +\nu}{b},k \right)
   \cs\left(2K\frac{\mu \tau +\nu}{b},k \right) \nonumber \\
\label{eq:452-2}
   & \quad =
   \frac{-a^{2}-b^{2}+2}{6ab}+\frac{-a^{2}+2b^{2}-1}{6ab}\lambda (\tau).
\end{align}
Taking the limit $\tau \to \sqrt{-1}\infty$, (\ref{eq:452-1}) and (\ref{eq:452-2}) degenerate to (\ref{eq:another prot thm1.1}) and (\ref{eq:2another prot thm1.1}) respectively.

\subsubsection{$2\,\,|\,\,b$, $2\not|\,\,a+b$}
\begin{align}
& \frac{1}{a}\!\!\!\!\!\sum_{\substack{\mu,\nu =0 \\ (\mu,\nu) \not=(0,0)}}^{a-1}\!\!\!\!\!(-1)^{\nu }
   \ds\left(2Kb\frac{\mu \tau +\nu}{a},k \right)
   \ns\left(2K\left(z+\frac{\mu \tau +\nu}{a}\right),k \right) \nonumber \\
& \quad +\frac{1}{b}\!\!\!\!\!\sum_{\substack{\mu,\nu =0 \\ (\mu,\nu) \not=(0,0)}}^{b-1}\!\!\!\!\!(-1)^{\mu +\nu }
   \ns\left(2Ka\frac{\mu \tau +\nu}{b},k \right)
   \ns\left(2K\left(z+\frac{\mu \tau +\nu}{b}\right),k \right) \nonumber \\
\label{eq:453-1}
   & \quad =
   -\ns(2Kaz,k )\ds(2Kbz,k )
   +\frac{1}{ab}\ds(2Kz,k )\cs(2Kz,k ), 
\end{align}
\begin{align}
& \frac{1}{a}\!\!\!\!\!\sum_{\substack{\mu,\nu =0 \\ (\mu,\nu) \not=(0,0)}}^{a-1}\!\!\!\!\!(-1)^{\nu }
   \ds\left(2Kb\frac{\mu \tau +\nu}{a},k \right)
   \ns\left(2K\frac{\mu \tau +\nu}{a},k \right) \nonumber \\
& \quad +\frac{1}{b}\!\!\!\!\!\sum_{\substack{\mu,\nu =0 \\ (\mu,\nu) \not=(0,0)}}^{b-1}\!\!\!\!\!(-1)^{\mu +\nu }
   \ns\left(2Ka\frac{\mu \tau +\nu}{b},k \right)
   \ns\left(2K\frac{\mu \tau +\nu}{b},k \right) \nonumber \\
\label{eq:453-2}
   & \quad =
   \frac{-a^{2}-b^{2}-1}{6ab}+\frac{-a^{2}+2b^{2}-1}{6ab}\lambda (\tau).
\end{align}
Taking the limit $\tau \to \sqrt{-1}\infty$, (\ref{eq:453-1}) and (\ref{eq:453-2}) degenerate to (\ref{eq:another prot thm1.2}) and (\ref{eq:2another prot thm1.2}) respectively.

\subsection{$(i,j)=(0,1)$, $(m,n)=(0,1)$}

\subsubsection{$2\not|\,\,a+b$}
\begin{align}
& \frac{1}{a}\!\!\!\!\!\sum_{\substack{\mu,\nu =0 \\ (\mu,\nu) \not=(0,0)}}^{a-1}\!\!\!\!\!(-1)^{\nu }
   \ns\left(2Kb\frac{\mu \tau +\nu}{a},k \right)
   \ns\left(2K\left(z+\frac{\mu \tau +\nu}{a}\right),k \right) \nonumber \\
& \quad +\frac{1}{b}\!\!\!\!\!\sum_{\substack{\mu,\nu =0 \\ (\mu,\nu) \not=(0,0)}}^{b-1}\!\!\!\!\!(-1)^{\nu }
   \ns\left(2Ka\frac{\mu \tau +\nu}{b},k \right)
   \ns\left(2K\left(z+\frac{\mu \tau +\nu}{b}\right),k \right) \nonumber \\
\label{eq:461-1}
   & \quad =
   -\ns(2Kaz,k )\ns(2Kbz,k )
   +\frac{1}{ab}\ds(2Kz,k )\cs(2Kz,k ), \\
& \frac{1}{a}\!\!\!\!\!\sum_{\substack{\mu,\nu =0 \\ (\mu,\nu) \not=(0,0)}}^{a-1}\!\!\!\!\!(-1)^{\nu }
   \ns\left(2Kb\frac{\mu \tau +\nu}{a},k \right)
   \ns\left(2K\frac{\mu \tau +\nu}{a},k \right) \nonumber \\
& \quad +\frac{1}{b}\!\!\!\!\!\sum_{\substack{\mu,\nu =0 \\ (\mu,\nu) \not=(0,0)}}^{b-1}\!\!\!\!\!(-1)^{\nu }
   \ns\left(2Ka\frac{\mu \tau +\nu}{b},k \right)
   \ns\left(2K\frac{\mu \tau +\nu}{b},k \right) \nonumber \\
\label{eq:461-2}
   & \quad =
   -\frac{a^{2}+b^{2}+1}{6ab}(1+\lambda (\tau)).   
\end{align}
Taking the limit $\tau \to \sqrt{-1}\infty$, (\ref{eq:461-1}) and (\ref{eq:461-2}) degenerate to (\ref{eq:another prot thm1.2}) and (\ref{eq:2another prot thm1.2}) respectively.


\bibliographystyle{amsplain}

\noindent 
Department of Mathematics, Graduate School of Science, Kobe University, \\
1-1, Rokkodai, Nada-ku, Kobe, 657-8501, JAPAN\\
E-mail: g-shibukawa@math.kobe-u.ac.jp
\end{document}